\documentclass[11pt,a4paper]{amsart}
\usepackage[T1]{fontenc}
\usepackage[utf8]{inputenc}
\usepackage{lmodern}
\usepackage[margin=27mm,headheight=14pt]{geometry}
\usepackage{amsmath,amssymb,mathtools}
\usepackage{tikz-cd}
\usepackage{enumitem}
\usepackage{microtype}
\usepackage{needspace}
\usepackage[hidelinks]{hyperref}
\hypersetup{pdftitle={Frobenius functors and n-torsionfree objects},
  pdfauthor={ Zhibing Zhao}}
\setlist[enumerate,1]{label=\textup{(\arabic*)},leftmargin=2em,itemsep=2pt,topsep=4pt}
\newtheorem{theorem}{Theorem}[section]
\newtheorem{proposition}[theorem]{Proposition}
\newtheorem{lemma}[theorem]{Lemma}
\newtheorem{corollary}[theorem]{Corollary}
\theoremstyle{definition}
\newtheorem{definition}[theorem]{Definition}
\newtheorem{example}[theorem]{Example}
\newtheorem{remark}[theorem]{Remark}
\newcommand{\A}{\mathcal A}
\newcommand{\B}{\mathcal B}
\newcommand{\Pj}{\mathcal P}
\newcommand{\T}{\mathcal T}
\newcommand{\mm}{\mathfrak m}
\newcommand{\Mod}{\text{-}\mathrm{Mod}}
\newcommand{\modu}{\text{-}\mathrm{mod}}
\newcommand{\op}{\mathrm{op}}
\DeclareMathOperator{\Hom}{Hom}
\DeclareMathOperator{\Ext}{Ext}
\DeclareMathOperator{\Ker}{Ker}
\DeclareMathOperator{\Coker}{Coker}
\DeclareMathOperator{\add}{add}
\DeclareMathOperator{\GP}{GP}
\DeclareMathOperator{\Id}{Id}
\DeclareMathOperator{\Tr}{Tr}
\DeclareMathOperator{\fd}{fd}
\DeclareMathOperator{\id}{id}
\DeclareMathOperator{\soc}{soc}
\DeclareMathOperator{\Arr}{Arr}
\DeclareMathOperator{\Gr}{Gr}
\newcommand{\xra}[1]{\xrightarrow{#1}}
\newcommand{\orth}[1]{{}^{\perp}\Pj(#1)}
\numberwithin{equation}{section}
\allowdisplaybreaks[1]

\title{Frobenius functors and $n$-torsionfree objects}
\author[Z. Zhao]{Zhibing Zhao}
\address{School of Mathematical Sciences,
Anhui University, Hefei 230601, Anhui, P. R. China}
\date{}
 \thanks{E-mail:zbzhao$\symbol{64}$ahu.edu.cn (Z. B. Zhao).}

\begin{document}
\begin{abstract}
We study $n$-torsionfree objects in abelian categories with enough projectives.
Frobenius functors preserve $n$-torsionfreeness, and faithful ones reflect it.
We prove that stabilization of the torsionfree filtration implies weak
Gorensteinness. For Frobenius extensions satisfying a generator condition,
we compare the terms of minimal injective resolutions and obtain transfer
of Auslander-type conditions and of the Auslander--Gorenstein conjecture.
We also compute a family of
non-Gorenstein algebras whose torsionfree filtrations stabilize at level
two and contain explicit nonprojective Gorenstein projective modules.
\end{abstract}
\maketitle

\section{Introduction}

Let $R$ be a two-sided noetherian ring. A finitely generated left
$R$-module $M$ is called $n$-torsionfree if
\[
  \Ext^i_{R^{\op}}(\Tr_R M,R)=0\qquad(1\leq i\leq n),
\]
where $\Tr_R M$ is the transpose of $M$. This notion was introduced by
Auslander and Bridger in \cite{AB}. The module $M$ is 1-torsionfree  if and only if it is torsionless, and $M$ is 2-torsionfree if and only if it is reflexive. In \cite{AB}, Auslander and Bridger described the $n$-torsionfreeness of a module in terms of $n$ successive monomorphic left projective approximations; see
\cite[Theorem 2.17]{AB}. This description allows one to define
$n$-torsionfree objects in an abelian category $\A$ with enough projectives; see Definition \ref{def:torsionfree}.

Denote by $\T^n(\A)$ the full subcategory of $\A$ consisting of all $n$-torsionfree objects. We obtain a filtration
\[
  \A=\T^0(\A)\supseteq\T^1(\A)\supseteq\T^2(\A)\supseteq\cdots.
\]

We first study the behavior of these subcategories under adjoint functors.
For a Frobenius extension, Zhao \cite[Theorem~3.5]{Zhao24} proved that
restriction preserves and reflects $n$-torsionfreeness of finitely generated
modules. Chen and Ren \cite{CR} proved the corresponding result for
Gorenstein projective objects under faithful Frobenius functors, including
Frobenius pairs with twists. Related results for relative torsionfreeness
and tensor functors appear in \cite{BLZ} and \cite[Lemma~5.2]{Liu}.
Our first result concerns exact adjoint pairs; see Theorem~\ref{thm:adjoint-transfer}.
We denote the projective objects of $\A$ by $\Pj(\A)$.

\vspace{0.2cm}

{\bf Theorem A}
Let $F:\A\rightleftarrows\B:G$ be an adjoint pair with $F$ left adjoint to $G$.
Assume that $F$ and $G$ are exact and that $G$ preserves projectives.
Then $F$ preserves $n$-torsionfree objects for every $n\geq1$.
If, moreover, $F$ is faithful and
\[
  \Pj(\A)=\add G(\Pj(\B)),
\]
then $X\in\T^n(\A)$ if and only if $F(X)\in\T^n(\B)$.

\vspace{0.2cm}

Frobenius functors satisfy the preservation hypotheses, and faithful
Frobenius functors also satisfy the condition on projectives.
Thus the theorem extends Zhao's result to abelian categories.
It also gives another proof of the Gorenstein projective case in
\cite{CR}. An example in Section~\ref{sec:functors} shows that the
adjoint pair need not be Frobenius, see Example~\ref{ex:arrow}. We also obtain a componentwise
criterion for bounded complexes, see Proposition~\ref{prop:complexes}.

 We next relate infinite torsionfreeness
to Gorenstein projectivity. Put$
 \orth{\A}=\{X\in\A\mid \Ext^i_{\A}(X,P)=0
       \text{ for all }i\geq1\text{ and }P\in\Pj(\A)\},
 \T^\infty(\A)=\bigcap_{n\geq1}\T^n(\A)$.
Similar to the case of module categories, we show that
\[
  \GP(\A)=\orth{\A}\cap\T^\infty(\A).
\]
Consequently, $\A$ is weakly Gorenstein if and only if every object of
$\orth{\A}$ is $1$-torsionfree. This is the categorical form of
\cite[Theorem~1.2]{RZ}. We further prove that an equality
$\T^n(\A)=\T^{n+1}(\A)$ at one level implies weak Gorensteinness.
For a Krull--Schmidt category, this applies whenever one level contains
only finitely many indecomposable isomorphism classes.

In Section~\ref{sec:extensions}, we apply these results to Frobenius
extension $S/R$. Under the assumption that $S_R$ is a generator,
we prove that corresponding terms in the minimal injective resolutions
of ${}_RR$ and ${}_SS$ have equal flat dimensions. This gives transfer
of quasi-$n$-Gorensteinness and the Auslander condition. It also shows that validity of the Auslander–Gorenstein conjecture is preserved under Frobenius extensions of Artin algebras, see Theorem \ref{prop:agc}.

\vspace{0.2cm}

{\bf Theorem B}
Let $S/R$ be a Frobenius extension of Artin algebras.
Assume that $S_R$ is a generator. Then AGC holds for $R$ if and only
if it holds for $S$.

\vspace{0.2cm}

Finally, we compute all torsionfree levels for
\[
 S_{d,m}=k[x_1,\ldots,x_d,t]/\big((x_1,\ldots,x_d)^2,t^m\big),
 \qquad d,m\geq2.
\]
These algebras are not Gorenstein. Their torsionfree filtrations have
two strict inclusions and then stabilize at the Gorenstein projective
modules. The modules $S_{d,m}/(t^q)$, for $1\leq q<m$, are explicit
nonprojective Gorenstein projective modules.

\medskip
\noindent\textbf{Notation.}
Throughout, abelian categories are locally small and have enough
projectives. Subcategories are full and closed under isomorphisms.
For a class $\mathcal X$, we write $\add\mathcal X$ for its closure
under finite direct sums and direct summands. Rings and ring
homomorphisms are unital.
We write $R\Mod$ for all left $R$-modules and $R\modu$ for finitely
generated left $R$-modules. Classical torsionfreeness is considered
over two-sided noetherian rings. Weak Gorensteinness of Artin algebras
refers to finitely generated modules. Right-handed statements are
interpreted over opposite rings.

\section{\texorpdfstring{$n$}{n}-torsionfree objects}\label{sec:filtration}

In this section, we introduce the notion of $n$-torsionfree object in an abelian category $\A$ with enough projectives, and some basic properties of this kind of objects are given.
\subsection{$n$-torsionfree objects}
Let $\A$ be an abelian category  with enough projective objects. The latter condition means that for each object $M$, there is an epimorphism $P\rightarrow M$ with $P$ projective.
A morphism $f:M\to P$ with $P\in\Pj(\A)$ is called a
\emph{left projective approximation} if every morphism from $M$ to a
projective object factors through $f$. We do not assume that such
approximations exist for every object.

\begin{definition}\label{def:torsionfree}
An object $M\in\A$ is \emph{$n$-torsionfree}, for $n\geq1$, if there
is an exact sequence
\begin{equation}\label{eq:approx-sequences}
 0\longrightarrow M\xra{f_1}P_1\xra{f_2}\cdots\xra{f_n}P_n
\end{equation}
with each $P_i$ projective, such that the inclusion
$\operatorname{Im}f_i\to P_i$ is a left projective approximation
for $1\leq i\leq n$. Their full subcategory is denoted by
$\T^n(\A)$. Set $\T^0(\A)=\A$ and
$\T^\infty(\A)=\bigcap_{n\geq1}\T^n(\A)$.
Objects of $\T^\infty(\A)$ are called \emph{infinitely torsionfree}.
\end{definition}

\begin{remark}\label{rem:classical}
(1) Every object in $\mathcal{A}$ is $0$-torsionfree and every projective object is infinitely torsionfree. If an object $M$ is $m$-torsionfree, then it is $n$-torsionfree when $m>n$. And we have a decreasing sequence of subcategories of $\mathcal{A}$
$$\mathcal{T}^0(\mathcal{A})\supseteq \mathcal{T}^1(\mathcal{A})\supseteq \cdots \supseteq \mathcal{T}^n(\mathcal{A})\supseteq\cdots.$$

(2) Let $\A$ be the left $R$-module category $R$-mod with $R$ a two-sided noetherian ring. Then an object $M$ which is $n$-torsionfree in $\mathcal{A}$  is just an $n$-torsionfree module; see \cite[Theorem 2.17]{AB}.

In particular, $\T^1(R\modu)$ consists of torsionless modules and
$\T^2(R\modu)$ consists of reflexive modules. For arbitrary modules,
we use Definition~\ref{def:torsionfree} without a transpose description.
\end{remark}

\begin{lemma}\label{lem:ext-characterization}
Let $M$ be an object in $\mathcal{A}$. Then $M$ is $n$-torsionfree if and only if it admits an exact
sequence
\begin{equation}\label{eq:finite-coresolution}
 0\longrightarrow M\xra{f_1}P_1\xra{f_2}\cdots
 \xra{f_n}P_n\longrightarrow T_n\longrightarrow0
\end{equation}
with all $P_i$ projective and $\Ext^j_{\A}(T_n,Q)=0$ for
$1\leq j\leq n$ and every $Q\in\Pj(\A)$.
\end{lemma}

\begin{proof}
 If $M$ is $n$-torsionfree, then there exists an exact sequence 
 $$0\longrightarrow M\xra{f_1}P_1\xra{f_2}\cdots
 \xra{f_n}P_n\longrightarrow T_n$$
 with $P_i$  projective, such that ${\rm Im} f_i \rightarrow P_i$ is a left $\mathcal{P(A)}$-approximation of ${\rm Im} f_i$ for $1\leq i\leq n$. Taking ${\rm Coker}f_i=T_i$, we have that ${\rm Ext}^1_{\mathcal{A}}(T_i, \mathcal{P(A)})=0$ for $1\leq i\leq n$. By dimension shifting,  ${\rm Ext}^j_{\mathcal{A}}(T_i, \mathcal{P(A)})\cong{\rm Ext}^{n-i+j}_{\mathcal{A}}(T_n, \mathcal{P(A)})$, and we obtain that ${\rm Ext}^{j}_{\mathcal{A}}(T_n, \mathcal{P(A)})=0$ for $1\leq j\leq n$.

Conversely, if there exists an exact sequence (2.2)
with $P_i$  projective and  ${\rm Ext}^{j}_{\mathcal{A}}(T_n, \mathcal{P(A)})=0$ for $1\leq j\leq n$, then ${\rm Ext}^1_{\mathcal{A}}(T_i, \mathcal{P(A)})\cong{\rm Ext}^{n-i+1}_{\mathcal{A}}(T_n, \mathcal{P(A)})=0$
for $1\leq i\leq n$ where $T_i={\rm Coker}f_i$. It follows that ${\rm Im} f_i \rightarrow P_i$ is a left $\mathcal{P(A)}$-approximation of ${\rm Im} f_i$ for $1\leq i\leq n$. Thus $M$ is $n$-torsionfree.
\end{proof}

\subsection{Comparison and recurrence}

It is well-known that the class of $n$-torsionfree modules over a two-sided Noetherian ring is closed under direct sums and direct summands. We will show that the subcategory $\mathcal{T}^n(\mathcal{A})$ has similar properties.  The following result seems to be well-known.
\begin{lemma}\label{lem:schanuel}
Suppose that
\[
 0\longrightarrow M\xra{f_1}P_1\xra{g_1}N_1\longrightarrow0
\]and 
\[
 0\longrightarrow M\xra{f_2}P_2\xra{g_2}N_2\longrightarrow0
\]
are approximation sequences with $P_i$ projective. Then
$P_1\oplus N_2\simeq P_2\oplus N_1$.
\end{lemma}

\begin{proof}Consider the following pushout diagram:

\[
\begin{tikzcd}[column sep=0.7cm,row sep=0.7cm]
& 0 \arrow[d] & 0 \arrow[d] \\
0 \arrow[r] & M \arrow[d,"f_2"] \arrow[r,"f_1"] & P_1 \arrow[d,dashed] \arrow[r] & N_1 \arrow[d,equals] \arrow[r] & 0 \\
0 \arrow[r] & P_2 \arrow[d] \arrow[r,dashed,"h_1"] & E \arrow[d] \arrow[r] & N_1 \arrow[r] & 0 \\
& N_2 \arrow[d] \arrow[r,equals] & N_2 \arrow[d] \\
& 0 & 0.
\end{tikzcd}
\]
Since $f_1$ is a left $\mathcal{P(A)}$-approximation of $M$, there exists a morphism $g:P_1\rightarrow P_2$ such that $f_2=g\cdot f_1$. By the universal property of pushout diagram, there exists a morphism $h:E\rightarrow P_2$ such that $h\cdot h_1={\rm Id}_{P_2}$. It follows that the third exact row is split, that is $E\cong P_2\oplus N_1$. Similarly, we can obtain that $E\cong P_1\oplus N_2$. Therefore, we have  $P_1\oplus N_2\cong P_2\oplus N_1$.
\end{proof}

\begin{proposition}\label{prop:summands} For any non-negative integer $n$, $\T^n(\A)$ is closed under finite direct sums and direct summands.
\end{proposition}

\begin{proof}
Since finite direct sums of approximation sequences are again approximation
sequences, we get that $\T^n(\A)$ is closed under finite direct sums. 

 For direct summands, we use induction on $n$. Let $M\cong M_1\oplus M_2$ be an $n$-torsionfree object in $\mathcal{A}$. We will show that $M_1$ is also $n$-torsionfree. For the case of $n=1$, there is a left $\mathcal{P(A)}$-approximation $\begin{tikzcd}[column sep=0.5cm]
0 \arrow[r] & M\cong M_1\oplus M_2\arrow[r,"f"] & P.
\end{tikzcd}$ Let $i$ be the canonical inclusion from $M_1$ to $M$ and corresponding projection.
Consider the monomorphism
$\begin{tikzcd}[column sep=0.5cm]
0 \arrow[r] & M_1\arrow[r,"f\cdot i"] & P.
\end{tikzcd}$
For any homomorphism $g_1: M_1\rightarrow P_1$ with $P_1$ projective, there exists a homomorphism $h:P\rightarrow P_1$ such that $g_1\cdot \pi=h\cdot f$ since $f$ is a left $\mathcal{P}(\mathcal{A})$-approximation of $M$.
\[
\begin{tikzcd}[column sep=0.7cm,row sep=0.7cm]
& P_1 & \\
0\arrow[r] & M_1 \arrow[u,"g_1"]\arrow[r,"f\cdot i"] & P\arrow[ul,dashed,"h"'] \\
& M \arrow[ur,"f"] \arrow[u,"\pi"] &
\end{tikzcd}.
\]
Then $g_1=g_1\cdot (\pi\cdot i)=h\cdot (f\cdot i)$, and we have $f\cdot i: M_1\rightarrow P$ is a left $\mathcal{P}(\mathcal{A})$-approximation of $M_1$. It follows that $M_1$ is $1$-torsionfree.

Now, we suppose that $n>1$ and $M\cong M_1\oplus M_2$ is $n$-torsionfree. Then there exists a short exact sequence
$\begin{tikzcd}[column sep=0.5cm]
0 \arrow[r] & M\cong M_1\oplus M_2\arrow[r,"f"] & P\arrow[r] & N\arrow[r] &0
\end{tikzcd}$
where $f$ is a left $\mathcal{P}(\mathcal{A})$-approximation and $N$ is $(n-1)$-torsionfree. By the proof of the case of $n=1$, there are two exact sequences
$\begin{tikzcd}[column sep=0.5cm]
0 \arrow[r] & M_1\arrow[r,"f\cdot i_1"] & P\arrow[r] & N_1\arrow[r] &0
\end{tikzcd}$
and
$\begin{tikzcd}[column sep=0.5cm]
0 \arrow[r] & M_2\arrow[r,"f\cdot i_2"] & P\arrow[r] & N_2\arrow[r] &0
\end{tikzcd}$,
such that $f\cdot i_1$ and $f\cdot i_2$ are left $\mathcal{P}(\mathcal{A})$-approximations of $M_1$ and $M_2$, respectively.
Then we have the following exact sequence

\[
\begin{tikzcd}[
  column sep=1.2cm,
  ampersand replacement=\&
]
0 \arrow[r]
  \& M_1\oplus M_2
  \arrow[
    r,
    description,
    "{\left(
      \begin{smallmatrix}
        f\cdot i_1 & 0\\
        0 & f\cdot i_2
      \end{smallmatrix}
    \right)}"
  ]
  \& P\oplus P
  \arrow[r]
  \& N_1\oplus N_2
  \arrow[r]
  \& 0
\end{tikzcd}
\]
with $\begin{pmatrix}f\cdot i_1 &0\\ 0& f\cdot i_2\end{pmatrix}$ a left $\mathcal{P}(\mathcal{A})$-approximation of $M$. It follows from Lemma \ref{lem:schanuel} that $P\oplus N_1\oplus N_2\cong P\oplus P\oplus N$. Since $P\oplus P\oplus N$ is $(n-1)$-torsionfree, we have $N_1$ is $(n-1)$-torsionfree by inductive hypothesis. Then there exists an exact sequence
$\begin{tikzcd}[column sep=0.5cm]
0 \arrow[r] & N_1\arrow[r,"f_1"] & P_1\arrow[r,"f_2"] & \cdots \arrow[r,"f_{n-1}"] & P_{n-1}
\end{tikzcd}$
with each $P_i$ projective, such that $\operatorname{Im} f_i \rightarrow P_i$ is a left $\mathcal{P}(\mathcal{A})$-approximation of $\operatorname{Im} f_i$ for $1\leq i\leq n-1$. Consider the exact sequence
$\begin{tikzcd}[column sep=0.5cm]
0 \arrow[r] & M_1 \arrow[r,"f"] & P\arrow[r,"f_1"] & P_1\arrow[r,"f_2"] & \cdots \arrow[r,"f_{n-1}"] & P_{n-1}
\end{tikzcd}$,
where $f:M_1(\cong\operatorname{Im}f)\rightarrow P$ is a left $\mathcal{P}(\mathcal{A})$-approximation and $\operatorname{Im} f_i \rightarrow P_i$ is a left $\mathcal{P}(\mathcal{A})$-approximation of $\operatorname{Im} f_i$ for $1\leq i\leq n-1$. Hence $M_1$ is $n$-torsionfree.
\end{proof}

If $\A$ has small coproducts and these preserve short exact sequences,
then $\T^n(\A)$ is also closed under arbitrary coproducts. Indeed,
coproducts of projectives are projective, and the approximation property
can be checked on each summand.

The following observation gives a recurrence relation of $n$-torsionfreeness.

\begin{proposition}\label{prop:recurrence}
 Let $\begin{tikzcd}[column sep=0.5cm]
0 \arrow[r] & M \arrow[r,"f"] & P \arrow[r] & N \arrow[r] & 0
\end{tikzcd}$ be a short exact sequence in $
\mathcal{A}$ with $f$ a left $\mathcal{P(A)}$-approximation of $M$ and let $n\geq 1$ be an integer.  Then $M$ is $n$-torsionfree if and only if $N$ is $(n-1)$-torsionfree.
In particular, $M$ is infinitely torsionfree if and only if $N$ is.
\end{proposition}

\begin{proof} Suppose that $M$ is $n$-torsionfree. The case of $n=1$ is trivial.

If $n\geq 2$, there exists an exact sequence $\begin{tikzcd}[column sep=0.5cm]
0 \arrow[r] & M \arrow[r,"f_1"] & P_1 \arrow[r,"f_2"] & \cdots \arrow[r,"f_n"] & P_n
\end{tikzcd}$ with $P_i\in\mathcal{P(A)}$, such that each ${\rm Im} f_i \rightarrow P_i$ is a left $\mathcal{P(A)}$-approximation of ${\rm Im} f_i$ for $1\leq i\leq n$.
Consider the short exact sequence $\begin{tikzcd}[column sep=0.5cm]
0 \arrow[r] & M \arrow[r,"f_1"] & P_1 \arrow[r] & N_1 \arrow[r] & 0
\end{tikzcd}$, where $f_1$ is a left $\mathcal{P(A)}$-approximation of $M$ and $N_1={\rm Coker}f_1$. It is easy to see that $N_1$ is $(n-1)$-torsionfree. By Lemma \ref{lem:schanuel}, we get $P_1\oplus N\cong P\oplus N_1$. Since $P\oplus N_1$ is $(n-1)$-torsionfree, we have that $N$ is also $(n-1)$-torsionfree by the proposition above.

 Conversely, prepend the given approximation sequence to a defining
sequence for the $(n-1)$-torsionfreeness of $N$. This gives
$n$-torsionfreeness of $M$ by Definition~\ref{def:torsionfree}.
Taking all $n$ yields the last assertion.
\end{proof}

\subsection{Objects of G-dimension zero}
As the origin of Gorenstein homological algebra, in the category of finitely generated modules over a two-sided noetherian ring, the module of G-dimension zero was introduced by Auslander and Bridger in \cite{AB}. We extend the definition to abelian categories with enough projectives.
\begin{definition} Let $\mathcal{A}$ be an abelian category with enough projective objects. An object $M$ is said to be of {\em G-dimension zero}, denoted by G-dim$(M)=0$, if the following conditions hold:
$(1)$ ${\rm Ext}^i_{\mathcal{A}}(M, \mathcal{P(A)})=0$ for all $i>0$; $(2)$ $M$ is $\infty$-torsionfree.
\end{definition}

Recall that an acyclic complex\[
\begin{tikzcd}[column sep=0.5cm]
P^\bullet = & \cdots \arrow[r] & P_1 \arrow[r] & P_0 \arrow[r] & P^0 \arrow[r] & P^1 \arrow[r] & \cdots
\end{tikzcd}
\]
of projective objects is said to be {\em totally acyclic}, provided it remains acyclic after applying ${\rm Hom}_{\mathcal{A}}(-, P)$ for any projective object $P\in\mathcal{A}$. An object $M\in\mathcal{A}$ is called {\em Gorenstein projective} if there is a totally acyclic complex $P^\bullet$ such that $M$ is isomorphic to its zeroth cocycle $Z^0(P^\bullet)$; see \cite{EJ}. We denote by $\GP(\A)$ the full subcategory of $\mathcal{A}$ consisting  of all Gorenstein projective objects.

Over a two-sided Noetherian ring, \cite[Theorem 4.2.6]{Chris} shows that a finitely generated module is Gorenstein projective if and only if it is a module of G-dimension zero. The following result generalizes this to a general abelian category with enough projective objects.

\begin{proposition}\label{prop:gp}
An object is Gorenstein projective if and only if it has G-dimension
zero. Equivalently,
\[
 \GP(\A)=\orth{\A}\cap\T^\infty(\A).
\]
\end{proposition}

\begin{proof}
A totally acyclic complex gives both a projective resolution which is
$\Hom_{\A}(-,Q)$-exact and approximation sequences of every finite
length. Thus each of its cocycles belongs to
$\orth{\A}\cap\T^\infty(\A)$.

Conversely, let $M=T_0\in\orth{\A}\cap\T^\infty(\A)$ and choose
an approximation sequence $0\to T_0\to P_1\to T_1\to0$.
Proposition~\ref{prop:recurrence} gives $T_1\in\T^\infty(\A)$.
The approximation property gives $\Ext^1_{\A}(T_1,Q)=0$, while
dimension shifting gives
\[
 \Ext^{j+1}_{\A}(T_1,Q)\simeq\Ext^j_{\A}(M,Q)=0
 \qquad(j\geq1).
\]
Hence $T_1\in\orth{\A}$. Repeating this construction produces
compatible approximation sequences
\[
 0\longrightarrow T_{i-1}\longrightarrow P_i\longrightarrow T_i
 \longrightarrow0\qquad(i\geq1).
\]
Splice them with a projective resolution of $M$. The resulting complex
is acyclic and remains so under $\Hom_{\A}(-,Q)$: on the right this
follows from the approximation property, and on the left from
$M\in\orth{\A}$ and dimension shifting. Thus $M$ is Gorenstein
projective.
\end{proof}

\section{Frobenius functors and torsionfreeness}\label{sec:functors}

We study preservation and reflection of torsionfreeness under exact
adjoint functors, and then specialize to Frobenius functors.
\subsection{Exact adjoint pairs}

Throughout this subsection, $F:\A\rightleftarrows\B:G$ denotes a pair
of exact additive functors, with $F$ left adjoint to $G$.
Then $F$ preserves projectives, since
$\Hom_{\B}(F(P),-)\simeq\Hom_{\A}(P,G(-))$ is exact for
$P\in\Pj(\A)$.

\begin{lemma}\label{lem:adjunction}
Let $F:\A\rightleftarrows\B:G$ be an exact adjoint pair, with $F$ left adjoint to $G$, and assume that $G(\Pj(\B))\subseteq\Pj(\A)$. Then $F$ preserves left projective
approximations and monomorphic left projective approximations.
For every $j\geq1$, one has
\[
 \Ext^j_{\A}(X,\Pj(\A))=0
 \quad\Longrightarrow\quad
 \Ext^j_{\B}(F(X),\Pj(\B))=0.
\]
This implication is an equivalence if
 $\Pj(\A)=\add G(\Pj(\B))$.

\end{lemma}

\begin{proof}
Let $a:X\to P$ be a left projective approximation. A morphism
$F(X)\to Q$, with $Q$ projective, corresponds under adjunction to
a morphism $X\to G(Q)$. The latter factors through $a$, since
$G(Q)$ is projective. Naturality of adjunction then gives a
factorization through $F(a)$. Exactness of $F$ preserves
monomorphisms.

Applying $F$ to a projective resolution and using adjunction gives
\begin{equation}\label{eq:ext-adjunction}
 \Ext^j_{\B}(F(X),Q)\simeq\Ext^j_{\A}(X,G(Q)).
\end{equation}
This proves the implication. If
 $\Pj(\A)=\add G(\Pj(\B))$,
additivity of $\Ext$ gives the converse.
\end{proof}

\begin{theorem}\label{thm:adjoint-transfer}
Let $F:\A\rightleftarrows\B:G$ be an exact adjoint pair, with $F$ left adjoint to $G$, and assume that $G(\Pj(\B))\subseteq\Pj(\A)$. Then $F$ preserves $n$-torsionfree
objects for all $n\geq1$. If $F$ is faithful and
$\Pj(\A)=\add G(\Pj(\B))$, then
\[
 X\in\T^n(\A)\quad\Longleftrightarrow\quad F(X)\in\T^n(\B).
\]
The same preservation and reflection statements hold for infinite
torsionfreeness.
\end{theorem}

\begin{proof}
Preservation follows by applying $F$ to approximation sequences and
using Lemma~\ref{lem:adjunction}.

For reflection, suppose that $F$ is faithful and
if
 $\Pj(\A)=\add G(\Pj(\B))$ holds. Let $F(X)\in\T^n(\B)$ and
choose an approximation sequence
\[
 0\longrightarrow F(X)\xra{u}Q\longrightarrow T\longrightarrow0,
 \qquad Q\in\Pj(\B),\quad T\in\T^{n-1}(\B).
\]
Write $\eta:\Id_{\A}\to GF$ and $\varepsilon:FG\to\Id_{\B}$ for
the unit and counit. The morphism adjoint to $u$ is
\[
 a=G(u)\eta_X:X\longrightarrow G(Q).
\]
Since $\varepsilon_Q F(a)=u$ is monic, $F(a)$ is monic. Exactness
and faithfulness of $F$ imply that $a$ is monic.

We show that $a$ is a left projective approximation. Let
$f:X\to G(Q')$ with $Q'$ projective, and let
$\widetilde f:F(X)\to Q'$ be its adjoint. Write $\widetilde f=vu$.
Then
\[
 f=G(\widetilde f)\eta_X=G(v)G(u)\eta_X=G(v)a.
\]
By $\Pj(\A)=\add G(\Pj(\B))$, every projective $P\in\A$ is a direct summand of some $G(Q')$.
Composing with the corresponding inclusion and projection proves the
same factorization for a map $X\to P$.

Set $X'=\Coker a$. By Lemma~\ref{lem:adjunction}, the sequence
\[
 0\longrightarrow F(X)\xra{F(a)}FG(Q)
 \longrightarrow F(X')\longrightarrow0
\]
is an approximation sequence. Lemma~\ref{lem:schanuel} gives
\begin{equation}\label{eq:cokernel-comparison}
 F(X')\oplus Q\simeq FG(Q)\oplus T.
\end{equation}
For $n=1$, the sequence defined by $a$ already proves the assertion.
For $n>1$, Proposition~\ref{prop:summands} and
\eqref{eq:cokernel-comparison} give $F(X')\in\T^{n-1}(\B)$.
Induction gives $X'\in\T^{n-1}(\A)$ and hence $X\in\T^n(\A)$.

Taking all $n$ proves the assertion for infinite torsionfreeness.
\end{proof}

\begin{corollary}\label{adjoint tranfer GP} Let $F:\A\rightleftarrows\B:G$ be an exact adjoint pair, with $F$ left adjoint to $G$, and assume that $G(\Pj(\B))\subseteq\Pj(\A)$. If $X\in\A$ is Gorenstein projective, then so is $F(X)$. Moreover, if $F$ is faithful and
$\Pj(\A)=\add G(\Pj(\B))$, then $X\in\mathcal{GP(A)} $ if and only if $F(X)\in \mathcal{GP(B)}$.
\end{corollary}
\begin{proof} This follows from
Lemma~\ref{lem:adjunction} and Proposition~\ref{prop:gp} and Theorem \ref{thm:adjoint-transfer}.
\end{proof}

\subsection{Frobenius pairs}
The Frobenius functor was introduced firstly by  Caenepeel, Militaru and Zhu in \cite{CMZ}. 
Let  $F:\mathcal{A}\rightarrow\mathcal{B}$ and $G:\mathcal{B}\rightarrow\mathcal{A}$ be two additive functors. Assume that $\alpha:\mathcal{A}\rightarrow\mathcal{A}$ and $\beta:\mathcal{B}\rightarrow\mathcal{B}$ are two autoequivalences. We say that $(F,G)$ is a {\em Frobenius pair} of type $(\alpha, \beta)$ between $\mathcal{A}$ and $\mathcal{B}$, provided that both $(F,G)$ and $(G,\beta F\alpha)$ are adjoint pairs. We call the functor $F$ a {\rm Frobenius functor}, if it fits into a Frobenius pair $(F,G)$. In this case, the functor $G$ is also a Frobenius functor. In other words, Frobenius functors always appear in pairs. For  basic properties of Frobenius pairs and Frobenius functors, see \cite{CDM,CGN,CR}.

By a {\em classical Frobenius pair} $(F,G)$, we mean a Frobenius pair of type $({\rm Id}_{\mathcal{A}},{\rm Id}_{\mathcal{B}})$. That is, both $(F,G)$ and $(G,F)$ are adjoint pairs.
It is clear that autoequivalences preserve projectives, Ext-orthogonality and all
torsionfree classes.

\begin{lemma}\label{lem:frobenius}
Frobenius functors are exact and preserve projectives and injectives.
If $F$ is faithful, then the unit $\eta:\Id_{\A}\to GF$ is
monomorphic and $\Pj(\A)=\add G(\Pj(\B))$.
\end{lemma}

\begin{proof}
Both functors have left and right adjoints, up to autoequivalences,
and hence are exact. Their adjoints are exact as well, so they
preserve projectives and injectives. The triangular identity makes
$F(\eta_X)$ split monic. Faithfulness and exactness of $F$ imply
that $\eta_X$ is monic.

Put $H=\alpha G\beta$. Then $(H,F)$ is an adjoint pair. If
$\delta:HF\to\Id_{\A}$ is the counit, $F(\delta_X)$ is split
epic, so $\delta_X$ is epic when $F$ is faithful.
For $P\in\Pj(\A)$, choose a projective epimorphism $Q\to F(P)$.
The composite $H(Q)\to HF(P)\xra{\delta_P}P$ is epic and splits.
Thus $\Pj(\A)=\add H(\Pj(\B))$. Applying $\alpha^{-1}$ and
using that both autoequivalences preserve projectives gives
$\Pj(\A)=\add G(\Pj(\B))$; compare \cite[Corollary~2.2]{CR}.
\end{proof}

\begin{corollary}\label{cor:frobenius-transfer}
Let $ F:\A\rightleftarrows\B:G$ with $(F,G)$ a Frobenius pair and $X\in\A$.

$(1)$ If $X\in\T^n(\A)$, then $ F(X)\in\T^n(\B)$ for any $n\geq 1$;

$(2)$ If $X\in\T^\infty(\A)$, then $ F(X)\in\T^{\infty}(\B)$;

$(3)$ If $X$ is Gorenstein projective in $\A$, then so is $F(X)$ in $\B$.

Moreover, if the Frobenius functor $F$ is faithful, then each implication in (1)-(3) is an equivalence.

That is, every Frobenius functor preserves finite and infinite torsionfreeness
and Gorenstein projectivity. A faithful Frobenius functor also
reflects these properties.
\end{corollary}

\begin{proof}
Apply Lemma~\ref{lem:frobenius} and Theorem~\ref{thm:adjoint-transfer}.
The argument applies to either member of a Frobenius pair by composing
the adjunctions with the defining autoequivalences.
\end{proof}

\begin{remark}(1) The Gorenstein projective assertion is \cite[Theorem~3.2]{CR}.
Here it follows from finite torsionfreeness and Ext-orthogonality.
For related results on relative Gorenstein objects, see \cite{LM}.

(2) The faithfulness is
needed for reflection; compare \cite[Remark~3.5]{CR}. The projection $\A\times\A'\to\A$ and the inclusion
$X\mapsto(X,0)$ form a classical Frobenius pair.
Take $\A=\A'=\mathbb Z\Mod$ and a prime $p$.
The object $(0,\mathbb Z/p\mathbb Z)$ is not $1$-torsionfree,
but its image under the projection is zero. 

(3) The following example shows that the functor in Theorem \ref{thm:adjoint-transfer} need not be Frobenius.
\end{remark}

\begin{example}\label{ex:arrow}
Let $\Arr(\A)$ be the abelian category of morphisms in $\A$, and set
\[
 F(X)=(0\to X),\qquad G(Y\to Z)=Z.
\]
Then $(F, G)$ is an adjoint pair, both functors are exact, and $F$ is faithful.
The projectives of $\Arr(\A)$ are the direct summands of objects
\[
 (P\xra{1}P)\oplus(0\to Q),\qquad P,Q\in\Pj(\A).
\]
Indeed, maps from these objects are determined by maps from $P$ to the
source and from $Q$ to the target. For an arrow $f:Y\to Z$, projective
epimorphisms $p:P\to Y$ and $q:Q\to Z$ give an epimorphism
$(p,(fp,q))$ from the arrow $P\to P\oplus Q$ onto $f$.
Thus $G$ preserves projectives, and $GF=\Id_{\A}$ gives
that $\Pj(\A)=\add G(\Pj(\Arr(\A)))$.

For $\A=k\modu$, this pair is not Frobenius. Indeed, the module $k$ is
injective, whereas $(0\to k)$ is not, since the sequence
\[
 0\longrightarrow(0\to k)\longrightarrow(k\xra{1}k)
 \longrightarrow(k\to0)\longrightarrow0
\]
does not split. But, it is well-known that Frobenius functors preserve injectives.
\end{example}

\subsection{Bounded complexes}

Let $\A$ be an abelian category with enough projective objects. Then the category $C^b(\A)$ of bounded cochain complexes over $\A$ is again an abelian category, with kernels, cokernels, and exact sequences computed degreewise. In particular, a sequence of complexes in $C^b(\A)$ is exact if and only if the corresponding sequence is exact in each degree.

\begin{proposition}\label{prop:complexes} Let $n\geq 1$. In the abelian category $C^b(\A)$ of bounded cochain complexes with componentwise exact sequences, one has \[ X^\bullet\in\T^n(C^b(\A)) \quad\Longleftrightarrow\quad X^i\in\T^n(\A)\text{ for every }i. \] Moreover, \[ X^\bullet\in\T^\infty(C^b(\A)) \quad\Longleftrightarrow\quad X^i\in\T^\infty(\A)\text{ for every }i, \] and \[ X^\bullet\in\GP(C^b(\A)) \quad\Longleftrightarrow\quad X^i\in\GP(\A)\text{ for every }i. \]
 \end{proposition} 

\begin{proof} Let \[ U:C^b(\A)\longrightarrow\Gr^b(\A) \] be the functor forgetting the differential, where $\Gr^b(\A)$ is the category of finitely supported graded objects. For $M\in\A$, denote by $D^i(M)$ the complex having $M$ in degrees $i$ and $i+1$, with identity differential between them. For $M^\bullet\in\Gr^b(\A)$, 
set \[ L(M^\bullet)=\bigoplus_iD^i(M^i), \qquad R(M^\bullet)=\bigoplus_iD^{i-1}(M^i). \] These direct sums are finite. 

A chain map $D^i(M)\to X^\bullet$ is uniquely determined by its component $M\to X^i$, and similarly a chain map $X^\bullet\to D^{i-1}(M)$ is uniquely determined by its component $X^i\to M$. Hence $(L,U)$ and $(U,R)$ are both adjoint pairs.

Let $\sigma$ be the shift on $\Gr^b(\A)$ defined by $(\sigma M)^i=M^{i+1}$. Then $R=L\sigma$. Since $(L,U)$ and $(R,\sigma^{-1}U)$ are both adjoint pairs, we have $(U,R)$ is a Frobenius pair of type $(\Id,\sigma^{-1})$. 

We next note that both $C^b(\A)$ and $\Gr^b(\A)$ have enough projectives. For $\Gr^b(\A)$ this follows componentwise. If $X^\bullet\in C^b(\A)$, choose projective epimorphisms $p^i:P^i\to X^i$. By the adjunction $L\dashv U$, the graded morphism $(p^i)_i:P^\bullet\to U(X^\bullet)$ induces a chain map \[ \bigoplus_iD^i(P^i)\longrightarrow X^\bullet. \] In degree $i$ its restriction to the summand $P^i$ is $p^i$, and hence this chain map is epic. Since each $D^i(P^i)$ is projective, $C^b(\A)$ has enough projectives.

 The functor $U$ is faithful. Thus Corollary~\ref{cor:frobenius-transfer} gives \[ X^\bullet\in\T^n(C^b(\A)) \quad\Longleftrightarrow\quad U(X^\bullet)\in\T^n(\Gr^b(\A)). \] 

It remains to describe torsionfreeness in $\Gr^b(\A)$. Projective objects and exact sequences in $\Gr^b(\A)$ are componentwise. Moreover, a morphism \[ f:M^\bullet\longrightarrow P^\bullet, \qquad P^\bullet\in\Pj(\Gr^b(\A)), \] is a left projective approximation if and only if each $f^i:M^i\to P^i$ is a left $\Pj(\A)$-approximation.

 Indeed, one direction follows by factoring morphisms componentwise, while the other can be checked by taking projective graded objects concentrated in a single degree. Consequently, \[ M^\bullet\in\T^n(\Gr^b(\A)) \quad\Longleftrightarrow\quad M^i\in\T^n(\A)\text{ for every }i. \] For the reverse implication, the defining approximation sequences of the finitely many nonzero components can be assembled degreewise in $\Gr^b(\A)$. This proves the first assertion. 

Taking the intersection over all $n\geq1$ gives \[ X^\bullet\in\T^\infty(C^b(\A)) \quad\Longleftrightarrow\quad X^i\in\T^\infty(\A)\text{ for every }i. \]

 Finally, Ext-orthogonality in $\Gr^b(\A)$ is also componentwise: \[ M^\bullet\in\orth{\Gr^b(\A)} \quad\Longleftrightarrow\quad M^i\in\orth{\A}\text{ for every }i. \] Combining this with the preceding description of $\T^\infty(\Gr^b(\A))$ and Proposition~\ref{prop:gp}, we obtain \[ M^\bullet\in\GP(\Gr^b(\A)) \quad\Longleftrightarrow\quad M^i\in\GP(\A)\text{ for every }i. \] Applying Corollary~\ref{cor:frobenius-transfer} to the faithful Frobenius functor $U$ gives the required assertion for Gorenstein projective objects. \end{proof}

\begin{remark}\label{rem:complexes}
A stalk complex on a nonzero projective object $P\in\Pj(\A)$ is Gorenstein projective in $C^b(\A)$ by Proposition~\ref{prop:complexes}, since all its components are Gorenstein projective in $\A$. It is, however, not projective in $C^b(\A)$. Indeed, projective complexes are direct summands of finite sums of projective disks $D^i(Q)$, and hence are contractible. On the other hand, a nonzero stalk complex cannot be contractible: its differential is zero, so a contracting homotopy would force $1_P=0$. Thus componentwise projectivity does not imply projectivity in $C^b(\A)$.
The Gorenstein projective part of Proposition~\ref{prop:complexes}
is \cite[Example~3.9]{CR}.

The support is allowed to vary in $C^b(\A)$. The assertion need not
hold in a fixed interval. For instance, complexes in $k\modu$
supported in degrees $0,1$ form $\Arr(k\modu)$. The object
$(k\to0)$ has projective terms but is not $1$-torsionfree, since
every subobject of a projective arrow is monic.
\end{remark}

\section{Weakly Gorenstein categories}\label{sec:weak}

 Ringel and Zhang  introduced weakly Gorenstein algebra as a generalization of Gorenstein algebras in \cite{RZ}. We extend this notion to  abelian categories with enough projective objects.

\begin{definition}\label{def:weak}
The category $\A$ is \emph{weakly Gorenstein} if
$\GP(\A)=\orth{\A}$. An Artin algebra $R$ is \emph{left weakly
Gorenstein} if $R\modu$ is weakly Gorenstein. The right-handed
definition uses $R^{\op}\modu$.
\end{definition}

For Artin algebras, this agrees with the definition in \cite{RZ},
where objects of $\orth{R\modu}$ are called semi-Gorenstein-projective
modules. The following is the categorical form of
\cite[Theorem~1.2]{RZ}.

\begin{proposition}\label{prop:first-level}
The following conditions are equivalent.
\begin{enumerate}
\item The category $\A$ is weakly Gorenstein.
\item Every object of $\orth{\A}$ is infinitely torsionfree.
\item For some $n\geq1$, every object of $\orth{\A}$ is $n$-torsionfree.
\item Every object of $\orth{\A}$ is $1$-torsionfree.
\end{enumerate}
\end{proposition}

\begin{proof}
Proposition~\ref{prop:gp} gives $(1)\Leftrightarrow(2)$, and
$(2)\Rightarrow(3)\Rightarrow(4)$ is clear.
Assume (4) and let $M\in\orth{\A}$. Choose an approximation
sequence $0\to M\to P\to N\to0$.
Its approximation property gives $\Ext^1_{\A}(N,Q)=0$ for every
projective $Q$. Dimension shifting gives
\[
 \Ext^{j+1}_{\A}(N,Q)\simeq\Ext^j_{\A}(M,Q)=0\qquad(j\geq1).
\]
Thus $N\in\orth{\A}$, and (4) applies again. Repeating the
construction proves that $M$ is infinitely torsionfree.
\end{proof}

Let $M$ be an object in $\mathcal{A}$. There is an epimorphism $\begin{tikzcd}[column sep=0.5cm]
P \arrow[r, "f"] & M \arrow[r] & 0
\end{tikzcd}$ with $P$ projective. The object $\Omega^1 M$ is the kernel of $f$, and $\Omega^tM$ with $t\geq 0$ are called the \emph {syzygy} objects of $M$. Higher syzygies are defined by iteration.

\begin{lemma}\label{lem:syzygy}
Let $M\in\orth{\A}$. The following conditions are equivalent:

$(1)$ $M$ is Gorenstein projective;

$(2)$ $\Omega^t M$ is infinitely torsionfree
for some $t\geq1$;

$(3)$ every syzygy of $M$ is infinitely torsionfree.
\end{lemma}

\begin{proof} Since $M\in\orth{\A}$, dimension shifting shows that every syzygy of $M$ also belongs to $\orth{\A}$. Indeed, for $i\geq0$, $j\geq1$, and $Q\in\Pj(\A)$, one has \[ \Ext^j_{\A}(\Omega^iM,Q) \simeq \Ext^{j+i}_{\A}(M,Q)=0. \] 

Consider the short exact sequences coming from a projective resolution of $M$, \[ 0\longrightarrow\Omega^{i+1}M\longrightarrow P_i \longrightarrow\Omega^iM\longrightarrow0, \qquad i\geq0. \]
 Since $\Omega^iM\in\orth{\A}$, we have \[ \Ext^1_{\A}(\Omega^iM,Q)=0 \qquad\text{for every }Q\in\Pj(\A). \] Then the monomorphism $\Omega^{i+1}M\to P_i$ is a left projective approximation. Thus each of the above short exact sequences is an approximation sequence. Proposition~\ref{prop:recurrence} now gives \[ \Omega^{i+1}M\in\T^\infty(\A) \quad\Longleftrightarrow\quad \Omega^iM\in\T^\infty(\A) \qquad(i\geq0). \]

 Consequently, \[ M\in\T^\infty(\A) \quad\Longleftrightarrow\quad \Omega^tM\in\T^\infty(\A) \] for some, equivalently for every, $t\geq1$. Finally, since $M\in\orth{\A}$, Proposition~\ref{prop:gp} gives \[ M\in\GP(\A) \quad\Longleftrightarrow\quad M\in\T^\infty(\A). \] Combining the last two equivalences proves the result. \end{proof}

\begin{theorem}\label{thm:stabilization}
If $\T^n(\A)=\T^{n+1}(\A)$ for some $n\geq1$, then
$\T^n(\A)=\T^\infty(\A)$ and $\A$ is weakly Gorenstein.
\end{theorem}

\begin{proof}
Let $M\in\T^n(\A)=\T^{n+1}(\A)$. A defining first step for
its $(n+1)$-torsionfreeness is an approximation sequence
$0\to M\to P\to N\to0$ with $N\in\T^n(\A)$.
The same argument applies to $N$. Iteration gives approximation
sequences of arbitrary length, so $M\in\T^\infty(\A)$.

Now let $M\in\orth{\A}$. The first $n$ steps of a projective
resolution give
\[
 0\longrightarrow\Omega^nM\longrightarrow P_n\longrightarrow\cdots
 \longrightarrow P_1\longrightarrow M\longrightarrow0.
\]
Lemma~\ref{lem:ext-characterization} gives $\Omega^nM\in\T^n(\A)$.
Thus $\Omega^nM$ is infinitely torsionfree, and
Lemma~\ref{lem:syzygy} implies that $M$ is Gorenstein projective.
\end{proof}

Recall that $\A$ is Krull--Schmidt if each object is a finite direct
sum of indecomposables with local endomorphism rings. For Artin algebras, the case $n=1$ of the following recovers the torsionless-finite
criterion in \cite[Section~3.6]{RZ}.

\begin{corollary}\label{cor:finite-type}
Suppose that $\A$ is Krull--Schmidt and that $\T^n(\A)$ has only
finitely many indecomposable isomorphism classes for some $n\geq1$.
Then $\A$ is weakly Gorenstein.
\end{corollary}

\begin{proof}
For $m\geq n$, the indecomposable isomorphism classes in $\T^m(\A)$
form a descending sequence of subsets of a finite set, so they
eventually stabilize. By Proposition~\ref{prop:summands} and the
Krull--Schmidt property, $\T^m(\A)=\T^{m+1}(\A)$ for some $m$.
Apply Theorem~\ref{thm:stabilization}.
\end{proof}

Recall that a subcategory $\mathcal{X}$ of $\mathcal{A}$ is called \emph{extension-closed}, if the middle term $X$ of any short exact sequence $0\rightarrow X'\rightarrow X\rightarrow X''\rightarrow 0$ is in $\mathcal{X}$, provided the end terms $X', X''$ are in $\mathcal{X}$.

\begin{proposition}\label{prop:category-transfer}
Let $F:\A\rightleftarrows\B:G$ be an adjoint pair of exact functors,
with $F$ left adjoint to $G$. Assume that $G(\Pj(\B))\subseteq\Pj(\A)$, $F$ is faithful, and $\Pj(\A)=\add G(\Pj(\B))$.
\begin{enumerate}
\item If $\B$ is weakly Gorenstein, then so is $\A$.
\item If $\T^n(\B)$ is extension closed, then so is $\T^n(\A)$.
\item If $\T^n(\B)=\T^{n+1}(\B)$, then
      $\T^n(\A)=\T^{n+1}(\A)$.
\end{enumerate}
In particular, these assertions hold for faithful Frobenius functors.
If both functors of a Frobenius pair are faithful, each property is
equivalent for the two categories.
\end{proposition}

\begin{proof} By Theorem~\ref{thm:adjoint-transfer} and Corollary~\ref{adjoint tranfer GP}, $F$ preserves and reflects $n$-torsionfreeness for every $n\geq1$, as well as infinite torsionfreeness and Gorenstein projectivity.

 For (1), let $M\in\orth{\A}$. By Lemma~\ref{lem:adjunction}, we have $ F(M)\in\orth{\B}$.  Since $\B$ is weakly Gorenstein,  $\orth{\B}=\GP(\B)$, and hence $F(M)\in\GP(\B)$. As $F$ reflects Gorenstein projectivity, Corollary~\ref{adjoint tranfer GP} gives $M\in\GP(\A)$. Thus $\orth{\A}\subseteq\GP(\A)$.  The reverse inclusion always holds by Proposition~\ref{prop:gp}. Therefore $\GP(\A)=\orth{\A}$, so $\A$ is weakly Gorenstein.

 For (2), suppose that $\T^n(\B)$ is extension closed, and consider a short exact sequence \[ 0\longrightarrow X\longrightarrow Y\longrightarrow Z\longrightarrow0 \] in $\A$ with $X,Z\in\T^n(\A)$. Since $F$ is exact, applying $F$ gives a short exact sequence \[ 0\longrightarrow F(X)\longrightarrow F(Y) \longrightarrow F(Z)\longrightarrow0 \] in $\B$. It follows from  Theorem~\ref{thm:adjoint-transfer} that $ F(X),F(Z)\in\T^n(\B)$. By the extension closure of $\T^n(\B)$, it follows that $F(Y)\in\T^n(\B)$. Since $F$ reflects $n$-torsionfreeness, we obtain $Y\in\T^n(\A)$. Hence $\T^n(\A)$ is extension closed. 

For (3), assume that $\T^n(\B)=\T^{n+1}(\B)$.  Since $\T^{n+1}(\A)\subseteq\T^n(\A)$ always holds, it remains to prove the reverse inclusion. Let $M\in\T^n(\A)$. By preservation, \[ F(M)\in\T^n(\B)=\T^{n+1}(\B). \] Reflection of $(n+1)$-torsionfreeness now gives $M\in\T^{n+1}(\A)$. Thus one has $\T^n(\A)=\T^{n+1}(\A)$. 

If $F$ is a faithful Frobenius functor, Lemma~\ref{lem:frobenius} shows that the hypotheses above are satisfied, so all three assertions apply. If both functors in a Frobenius pair are faithful, the same argument may be applied in both directions, and hence each of the three properties is equivalent for $\A$ and $\B$. \end{proof}

\section{Applications to Frobenius extensions}\label{sec:extensions}

In this section, we apply the preceding results to Frobenius extensions. We first study the transfer of torsionfree classes and then compare the corresponding terms in minimal injective resolutions over the base ring and the extension ring.

\subsection{Transfer of torsionfree classes}

A \emph {ring extension} $S/R$ is a ring homomorphism $R\to S$, which is unital.
A ring extension $S/R$ is \emph{Frobenius} if ${}_RS$ is finitely
generated projective and
\begin{equation}\label{eq:frobenius-ring}
 {}_SS_R\simeq\Hom_R({}_RS,{}_RR)
\end{equation}
as $(S,R)$-bimodules. Equivalently, induction $I=S\otimes_R-$ and
restriction $U={\rm Hom}_S({_SS_R,-})$ form a classical Frobenius pair.
Symmetrically, $S_R$ is finitely generated projective and
${}_RS_S\simeq\Hom_{R^{\op}}(S_R,R_R)$; see \cite{CR,Kadison}.
For two-sided noetherian rings, these functors restrict to finitely
generated module categories.

\begin{lemma}\label{lem:generator}
For a Frobenius extension $S/R$, the following conditions are
equivalent:  

$(1)$ $S_R$ is a generator;

$(2)$ ${}_RS$ is a generator;

$(3)$ $S\otimes_R-$ is faithful. 

If these equivalent conditions hold, then
 $R_R\in\add S_R$ and $ {}_RR\in\add{}_RS$.
\end{lemma}

\begin{proof}
A finitely generated projective module is a generator if and only if
its additive closure contains the regular module. Duality of finite
projectives and the Frobenius isomorphisms give
\[
 R_R\in\add S_R
 \quad\Longleftrightarrow\quad
 {}_RR\in\add\Hom_{R^{\op}}(S_R,R_R)=\add{}_RS.
\]
Moreover, the natural isomorphism
\[
 S\otimes_R N\simeq
 \Hom_R\big(\Hom_{R^{\op}}(S_R,R_R),N\big)
\]
shows that induction is faithful exactly when ${}_RS$ is a generator.
\end{proof}

\begin{corollary}\label{cor:ring-transfer}
Let $S/R$ be Frobenius and $n\geq1$.
\begin{enumerate}
\item For every $M\in S\Mod$,
\[
 M\in\T^n(S\Mod)\quad\Longleftrightarrow\quad
 {}_RM\in\T^n(R\Mod).
\]
\item Induction preserves $n$-torsionfreeness. If $S_R$ is a
generator, then, for every $N\in R\Mod$,
\[
 N\in\T^n(R\Mod)\quad\Longleftrightarrow\quad
 S\otimes_RN\in\T^n(S\Mod).
\]
\end{enumerate}
Both assertions hold for infinite torsionfreeness and Gorenstein
projectivity. Over two-sided noetherian rings, they also hold for
finitely generated modules, with classical torsionfreeness.
\end{corollary}

\begin{proof}
The restriction  $U={\rm Hom}_S({_SS_R,-})$ is faithful. Apply Corollary~\ref{cor:frobenius-transfer}
and Lemma~\ref{lem:generator}.
\end{proof}

For finitely generated modules, (1) is \cite[Theorem~3.5]{Zhao24}.
Its Gorenstein projective counterpart appears in \cite{CR,Ren,Zhao19}.

\begin{corollary}\label{cor:weak-ring}
Let $S/R$ be a Frobenius extension of Artin algebras.
Left weak Gorensteinness, extension closure of $\T^n$, and the
equality $\T^n=\T^{n+1}$ in finitely generated module categories pass from $R$
to $S$. If $S_R$ is a generator, each property is equivalent for
$R$ and $S$. The right-handed assertions also hold.
\end{corollary}

\begin{proof} Since the restriction functor $U={\rm Hom}_S({_SS_R,-})$ is faithful, the first assertion follows from Proposition~\ref{prop:category-transfer}. Under the generator hypothesis, the induction functor $I=S\otimes_R-$ is also faithful, and each property is reflective by Proposition~\ref{prop:category-transfer} again.
Using Lemma~\ref{lem:generator} for the opposite rings, the right-handed assertions hold.
\end{proof}

\subsection{Minimal injective resolutions}
For a left $R$-module $M$, let $E_R^i(M)$ be the $i$th term in its
minimal injective resolution, with $E_R^0(M)$ the injective envelope.
Dimensions may be infinite, and we set $\fd_R0=0$.

\begin{lemma}\label{lem:minimal-injectives}
Over any ring, $E_R^i(M)$ is a direct summand of the $i$th term in
every injective resolution of $M$. Minimal injective resolutions
commute with finite direct sums. Consequently,
\[
 M\in\add N\quad\Longrightarrow\quad E_R^i(M)\in\add E_R^i(N).
\]
\end{lemma}

\begin{proof}
Let $0\to M\to I^0\xra{d^0}I^1\to\cdots$ be an injective
resolution. Essentiality of $M\subseteq E_R^0(M)$ shows that an
extension $E_R^0(M)\to I^0$ of $M\to I^0$ is monic and splits.
Write $I^0=E_R^0(M)\oplus K$, with $M$ in the first summand.
Then $d^0|_K$ is monic, and its image is injective. Removing the
split complex $K\xra{\sim}d^0(K)$ leaves an injective resolution
of $E_R^0(M)/M$. Repeating the argument proves the first assertion.

Finite direct sums of essential embeddings are essential, so
injective envelopes commute with finite direct sums. Applying this
successively to cokernels gives the same assertion in every degree.
The last statement follows by taking direct summands.
\end{proof}

\begin{lemma}\label{lem:flat-transfer}
For a Frobenius extension $S/R$, the induction functor $I=S\otimes_R-$ and the restriction $U={\rm Hom}_S({_SS_R,-})$
are exact and preserve injectives, projectives and flat modules.
In particular,
\[
 \fd_S(S\otimes_RN)\leq\fd_RN,\qquad
 \fd_R({}_RM)\leq\fd_SM.
\]
\end{lemma}

\begin{proof} Since $S/R$ is a Frobenius extension, both ${}_RS$ and $S_R$ are finitely generated projective. Moreover, the induction functor $I=S\otimes_R-$ and the restriction $U={\rm Hom}_S({_SS_R,-})$ form a classical Frobenius pair. Hence both functors are exact and preserve projective and injective modules; see Lemma~3.4.
We only need to verify the assertion concerning flat modules. 

Let $N$ be a flat left $R$-module. By Lazard's theorem, $N$ can be written as a filtered colimit $ N\cong \varinjlim_{\lambda} F_\lambda$,  where each $F_\lambda$ is a finitely generated free left $R$-module. Since tensor products commute with filtered colimits, we have \[ S\otimes_R N \cong S\otimes_R\Bigl(\varinjlim_{\lambda}F_\lambda\Bigr) \cong \varinjlim_{\lambda}(S\otimes_R F_\lambda). \] 

If $F_\lambda\cong R^{n_\lambda}$, then $S\otimes_R F_\lambda \cong S^{n_\lambda}$ as left $S$-modules. Thus every $S\otimes_R F_\lambda$ is a finitely generated free, and hence flat, left $S$-module. It follows again from Lazard's theorem that $S\otimes_R N$ is flat over $S$. Therefore the induction functor $I=S\otimes_R-$ preserves flat modules. Now let $M$ be a flat left $S$-module. Again by Lazard's theorem, $ M\cong\varinjlim_{\mu} G_\mu$,  where each $G_\mu$ is a finitely generated free left $S$-module. The restriction functor $U={\rm Hom}_S({_SS_R,-})$ commutes with filtered colimits, so ${}_RM\cong\varinjlim_{\mu}{}_RG_\mu$. 

 If $G_\mu\cong S^{m_\mu}$, then, after restriction of scalars, ${}_RG_\mu\cong({}_RS)^{m_\mu}$. Since ${}_RS$ is finitely generated projective over $R$, each ${}_RG_\mu$ is a finitely generated projective, hence flat, left $R$-module. Thus ${}_RM$ is flat over $R$. Hence restriction functor also preserves flat modules.

 Finally, let \[ \cdots\longrightarrow F_1\longrightarrow F_0 \longrightarrow N\longrightarrow 0 \] be a flat resolution of $N$ over $R$. Since induction functor is exact and preserves flat modules, applying $S\otimes_R-$ gives a flat $S$-resolution \[ \cdots\longrightarrow S\otimes_R F_1 \longrightarrow S\otimes_R F_0 \longrightarrow S\otimes_R N\longrightarrow 0. \] Consequently, $\operatorname{fd}_S(S\otimes_R N) \leq \operatorname{fd}_R N$. Similarly, if \[ \cdots\longrightarrow H_1\longrightarrow H_0 \longrightarrow M\longrightarrow 0 \] is a flat resolution of $M$ over $S$, then restricting scalars gives a flat $R$-resolution of ${}_RM$. Therefore  $\operatorname{fd}_R({}_RM) \leq \operatorname{fd}_S M$. This proves the lemma. 
\end{proof}

\begin{theorem}\label{thm:injective-comparison}
Let $S/R$ be a Frobenius extension. For all $i\geq0$, $N\in R\Mod$ and
$M\in S\Mod$, one has
\begin{align}
 \fd_S E_S^i(S\otimes_RN)&\leq\fd_R E_R^i(N),
       \label{eq:induced-injective}\\
 \fd_R E_R^i({}_RM)&\leq\fd_S E_S^i(M).
       \label{eq:restricted-injective}
\end{align}
If $S_R$ is a generator, then
\begin{equation}\label{eq:regular-comparison}
 \fd_SE_S^i(S)=\fd_RE_R^i(R)\quad(i\geq0),\qquad
 \id_SS=\id_RR.
\end{equation}
If $R\to S$ splits as an $R$-bimodule map, equality holds in
\eqref{eq:induced-injective} for every $N$ and $i$.
All assertions have right-handed versions.
\end{theorem}

\begin{proof}
Applying the induction functor $S\otimes_R-$ to a minimal injective resolution of ${_RN}$ gives
an injective resolution of $_SS\otimes_RN$.
Lemmas~\ref{lem:minimal-injectives} and~\ref{lem:flat-transfer} give
\[
 \fd_SE_S^i(S\otimes_RN)
 \leq\fd_S(S\otimes_RE_R^i(N))\leq\fd_RE_R^i(N).
\]
Applying the restriction functor $U={\rm Hom}_S({_SS_R,-})$  to a minimal injective resolution of ${_SM}$ gives
an injective resolution of $_RM$. Then we obtain \eqref{eq:restricted-injective} in the same way.

If $S_R$ is a generator, Lemma~\ref{lem:generator} gives
${}_RR\in\add{}_RS$. Hence
\[
 \fd_RE_R^i(R)\leq\fd_RE_R^i({}_RS)
 \leq\fd_SE_S^i(S)\leq\fd_RE_R^i(R).
\]
The same summand relation also gives
\[
 \id_SS\leq\id_RR\leq\id_R({}_RS)\leq\id_SS.
\]
This proves \eqref{eq:regular-comparison}, including infinite values.

If $\rho:S\to R$ is an $R$-bimodule retraction, then
$N\to{}_R(S\otimes_RN)$, $n\mapsto1\otimes n$, splits via
$s\otimes n\mapsto\rho(s)n$. Thus
\[
 \fd_RE_R^i(N)\leq\fd_RE_R^i({}_R(S\otimes_RN))
 \leq\fd_SE_S^i(S\otimes_RN).
\]
Together with \eqref{eq:induced-injective}, this gives equality.
For right modules, pass to opposite rings and use
Lemma~\ref{lem:generator}.
\end{proof}

\begin{remark}\label{rem:splitting}
For a fixed $N$, equality in \eqref{eq:induced-injective} only requires
$N\in\add{}_R(S\otimes_RN)$. A bimodule splitting ensures this for
every $N$. A one-sided splitting of regular modules does not by itself
supply the required left $R$-linear splitting after tensoring.
\end{remark}

\begin{remark}\label{rem:generator-needed}
The generator hypothesis in \eqref{eq:regular-comparison} cannot be
omitted for general Frobenius homomorphisms. The projection
$R=k\times\mathbb Z\to S=k$ is Frobenius, but $S_R$ is not a
generator. Since $E_R^1(R)=(0,\mathbb Q/\mathbb Z)$, one has
\[
 \id_RR=1\ne0=\id_SS,\qquad
 \fd_RE_R^1(R)=1\ne0=\fd_SE_S^1(S).
\]
Here $\mathbb Q/\mathbb Z$ has flat dimension one over $\mathbb Z$.
\end{remark}

\subsection{Auslander-type conditions and AGC}
For a two-sided noetherian ring $R$, we use the convention that $R$
is \emph{left quasi-$n$-Gorenstein} if
$\fd_RE_R^i(R)\leq i+1$ for $0\leq i<n$.
By Theorem 4.7 in \cite{AR}, $R$ is left quasi-$n$-Gorenstein if and only if $\T^i(R^{\rm op}$-mod) is extension-closed for $1\leq i\leq n$. The right quasi-$n$-Gorenstein condition is defined over $R^{\op}$.

 A two-sided noetherian ring $R$ satisfies the \emph{Auslander condition} if
\[
 \fd_RE_R^i(R)\leq i\qquad(i\geq0).
\]
The  condition is left--right symmetric; see \cite{Huang24}.

An \emph{Iwanaga--Gorenstein} ring is a two-sided noetherian ring of
finite self-injective dimension on both sides. For Artin algebras,
the flat dimensions displayed above equal projective dimensions.

\begin{corollary}\label{cor:auslander}
Let $S/R$ be a Frobenius extension of two-sided noetherian rings.
\begin{enumerate}
\item Left or right quasi-$n$-Gorensteinness and the Auslander
condition pass from $R$ to $S$.
\item If $S_R$ is a generator, these properties are equivalent for
$R$ and $S$. More generally, prescribed bounds on the flat dimensions
of corresponding terms in the regular injective resolutions are equivalent.
\item If $S_R$ is a generator, $R$ is Iwanaga--Gorenstein if and only
if $S$ is. Their corresponding left and right self-injective dimensions
are equal.
\end{enumerate}
\end{corollary}

\begin{proof}
Apply Theorem~\ref{thm:injective-comparison} on both sides.
\end{proof}

The \emph{Auslander--Gorenstein conjecture} (AGC) states that an Artin
algebra satisfying the Auslander condition is Iwanaga--Gorenstein.
We say that AGC holds for an Artin algebra $A$ if this implication
holds for $A$. This is distinct from the Auslander--Reiten
self-extension conjecture. The next result follows from
Corollary~\ref{cor:auslander}.

\Needspace{12\baselineskip}
\begin{theorem}\label{prop:agc}
Let $S/R$ be a Frobenius extension of Artin algebras.
Assume that $S_R$ is a generator. Then AGC holds for $R$ if and only
if it holds for $S$.
\end{theorem}

\begin{proof}
Suppose that AGC holds for $R$ and that $S$ satisfies the Auslander
condition. By Corollary~\ref{cor:auslander}(2), $R$ also satisfies
the Auslander condition and is therefore Iwanaga--Gorenstein.
Part (3) then implies that $S$ is Iwanaga--Gorenstein.
Thus AGC holds for $S$.

Conversely, suppose that AGC holds for $S$ and that $R$ satisfies
the Auslander condition. By Corollary~\ref{cor:auslander}(1), $S$
satisfies that condition, so it is Iwanaga--Gorenstein. Part (3)
implies that $R$ is Iwanaga--Gorenstein. Hence AGC holds for $R$.
\end{proof}

\begin{remark}\label{rem:literature}
(1) The ascent of quasi-$n$-Gorensteinness in
Corollary~\ref{cor:auslander}(1) is \cite[Theorem~4.1]{Zhao24};
ascent of $n$-Gorensteinness is proved in \cite[Proposition~5.2]{Zhao24}.
For noetherian algebras, Liu proves that quasi-$n$-Gorensteinness
is equivalent under faithfully flat Frobenius extensions
\cite[Corollary~5.19]{Liu}; \cite[Proposition~5.21]{Liu} treats $n$-Gorensteinness
under additional commutativity and centrality hypotheses.

(2) Xi \cite{Xi} compares flat-dominant dimensions using Frobenius
bimodules, and Gu, Huang and Zhao \cite{GHZ} study homological
invariants and $(l,n)$-conditions under Frobenius extensions.
Theorem~\ref{thm:injective-comparison} records the flat dimension of
each individual injective term, rather than only the length of an
initial segment of flat injectives. Its proof uses the established principle of comparing a minimal
injective resolution with another injective resolution.

(3) For an Artin algebra satisfying the Auslander condition,
Huang \cite[Corollary~4.12]{Huang24} identifies Iwanaga--Gorensteinness
with left or right weak Gorensteinness. This connects
Corollary~\ref{cor:weak-ring} with the Auslander-type comparisons.
Proposition~\ref{prop:agc} is a transfer consequence of the equivalence
of the Auslander condition and of Iwanaga--Gorensteinness for $R$ and
$S$ under the stated generator hypothesis.

\end{remark}

\section{Examples and explicit torsionfree levels}\label{sec:examples}

We compute the torsionfree filtration of a radical-square-zero
algebra and then apply Frobenius transfer to a truncated polynomial
extension. This gives non-Gorenstein algebras with
$\T^2=\T^\infty=\GP$ and nonprojective Gorenstein projective modules.

\begin{example}\label{ex:polynomial}
Let $R$ be a ring and $S=R[t]/(t^m)$, where $m\geq2$ and $t$ is
central. The coefficient map
\[
 E:S\longrightarrow R,\qquad
 E\left(\sum_{i=0}^{m-1}r_it^i\right)=r_{m-1},
\]
together with the dual bases $(t^i)_{i=0}^{m-1}$ and
$(t^{m-1-i})_{i=0}^{m-1}$, defines a Frobenius system. Then $
 E:S\longrightarrow R$ is a Frobenius extension.
Indeed, for every $s\in S$,
\[
 s=\sum_{i=0}^{m-1}t^iE(t^{m-1-i}s)
  =\sum_{i=0}^{m-1}E(st^i)t^{m-1-i}.
\]
Both ${}_RS$ and $S_R$ are free of rank $m$, and the constant-term
map splits $R\to S$ as $R$-bimodules. Thus the results of
Section~\ref{sec:extensions} apply, with their stated finiteness
hypotheses.
\end{example}

\begin{proposition}\label{prop:base-algebra} Let $k$ be a field, $d\geq 2$, and \[ R_d=k[x_1,\ldots,x_d]/(x_1,\ldots,x_d)^2. \] 
Let $\mathfrak m=(x_1,\ldots,x_d)$ be the unique maximal ideal of $R_d$, and identify $k$ with the simple $R_d$-module $R_d/\mathfrak m$. Then \[ \mathcal T^1(R_d\text{-mod}) =\operatorname{add}(R_d\oplus k), \] \[ \mathcal T^n(R_d\text{-mod}) =\mathcal P(R_d\text{-mod}) \qquad (n\geq2), \] and \[ {}^\perp\mathcal P(R_d\text{-mod}) =\mathcal P(R_d\text{-mod}). \] In particular, $R_d$ is weakly Gorenstein but not Gorenstein. \end{proposition}

 \begin{proof} Notice first that $R_d$ is a commutative local Artin algebra with \[ \mathfrak m^2=0,\qquad R_d/\mathfrak m\simeq k,\qquad \mathfrak m\simeq k^d \] as $R_d$-modules. We first determine the torsionless modules.

 Since $R_d$ is local, every finitely generated projective $R_d$-module is free. Thus a torsionless module may be viewed as a submodule $ N\subseteq R_d^b $ for some $b>0$. If $N\subseteq\mathfrak mR_d^b$, then $\mathfrak mN=0$ because $\mathfrak m^2=0$. Hence $N$ is semisimple, and therefore $ N\simeq k^c$ for some $c\geq0$. Suppose that $N\not\subseteq\mathfrak mR_d^b$. Then $N$ contains a vector $v=(v_1,\ldots,v_b)$ with a unit coordinate, say $v_j$. Projection onto the $j$th coordinate, followed by multiplication by $v_j^{-1}$, gives a homomorphism $ \pi:R_d^b\longrightarrow R_d $ such that $\pi(v)=1$. 

Hence $ R_d^b=R_dv\oplus\ker\pi $ and $ N=R_dv\oplus(N\cap\ker\pi)$. Moreover, $R_dv\simeq R_d$ and $\ker\pi\simeq R_d^{b-1}$. Induction on $b$ therefore gives $ N\simeq R_d^a\oplus k^c $ for some $a,c\geq0$. Conversely, $k$ embeds into $R_d$ via \[ k\longrightarrow R_d,\qquad \overline{1}\longmapsto x_1, \] since $\mathfrak m x_1=0$. Thus both $R_d$ and $k$ are torsionless. We conclude that \[ \mathcal T^1(R_d\text{-mod}) =\operatorname{add}(R_d\oplus k). \] 

We next determine the reflexive modules. Put $(-)^*=\operatorname{Hom}_{R_d}(-,R_d)$. Since an $R_d$-homomorphism $k\to R_d$ is determined by an element annihilated by $\mathfrak m$, we have $k^*\simeq\operatorname{soc}R_d$. As $\mathfrak m^2=0$, while no unit is annihilated by $\mathfrak m$, $\operatorname{soc}R_d=\mathfrak m$. Hence $ k^*\simeq\mathfrak m\simeq k^d$ and therefore $k^{**}\simeq (k^d)^*\simeq k^{d^2}$.  Since $d\geq2$, the module $k$ is not reflexive. Every torsionless module has the form $R_d^a\oplus k^c$, and the canonical evaluation morphism commutes with finite direct sums. Since $R_d$ is reflexive and $k$ is not, such a module is reflexive if and only if $c=0$. Thus \[ \mathcal T^2(R_d\text{-mod}) =\mathcal P(R_d\text{-mod}). \]

 Since projective modules are infinitely torsionfree and $\mathcal T^{n+1}\subseteq\mathcal T^n$, it follows that \[ \mathcal T^n(R_d\text{-mod}) =\mathcal P(R_d\text{-mod}) \qquad (n\geq2). \] It remains to compute ${}^\perp\mathcal P(R_d\text{-mod})$.

 Let $N$ be a nonprojective finitely generated $R_d$-module and take a projective cover \[ 0\longrightarrow\Omega N\longrightarrow P\longrightarrow N \longrightarrow0. \] Since $\Omega N\subseteq\operatorname{rad}P=\mathfrak mP$ and $\mathfrak m^2=0$, the module $\Omega N$ is semisimple. Moreover, $\Omega N\neq0$ because $N$ is nonprojective. Hence $\Omega N\simeq k^a$ for some $a\geq1$. Consider the exact sequence \[ 0\longrightarrow\mathfrak m \longrightarrow R_d \longrightarrow k \longrightarrow0. \] Applying $\operatorname{Hom}_{R_d}(-,R_d)$ gives \[ \operatorname{Ext}^1_{R_d}(k,R_d) \simeq \operatorname{Coker}\!\left( \operatorname{Hom}_{R_d}(R_d,R_d) \longrightarrow \operatorname{Hom}_{R_d}(\mathfrak m,R_d) \right). \] Since $\operatorname{soc}R_d=\mathfrak m$ and $\mathfrak m^2=0$, \[ \operatorname{Hom}_{R_d}(\mathfrak m,R_d) \simeq\operatorname{Hom}_k(\mathfrak m,\mathfrak m), \] which has $k$-dimension $d^2$. The image of $\operatorname{Hom}_{R_d}(R_d,R_d)$ consists precisely of scalar multiplications on $\mathfrak m$, and hence has dimension one. Therefore \[ \dim_k\operatorname{Ext}^1_{R_d}(k,R_d)=d^2-1>0. \] 

By dimension shifting, $\operatorname{Ext}^2_{R_d}(N,R_d) \simeq \operatorname{Ext}^1_{R_d}(\Omega N,R_d) \simeq \operatorname{Ext}^1_{R_d}(k,R_d)^a \neq0$. Thus no nonprojective module belongs to ${}^\perp\mathcal P(R_d\text{-mod})$, and hence \[ {}^\perp\mathcal P(R_d\text{-mod}) =\mathcal P(R_d\text{-mod}). \] Since the projective modules are Gorenstein projective, this also shows that $R_d$ is weakly Gorenstein. Finally, \[ \operatorname{soc}R_d=\mathfrak m \qquad\text{and}\qquad \dim_k\operatorname{soc}R_d=d>1. \] A commutative local Artin algebra is Gorenstein if and only if its socle is one-dimensional over its residue field. Hence $R_d$ is not Gorenstein. \end{proof}

\begin{corollary}\label{cor:family}
For $R_d$ as above and $m\geq2$, put $S=S_{d,m}=R_d[t]/(t^m)$.
Then
\begin{align*}
 \T^1(S\modu)&=\{M\mid {}_{R_d}M\in\add(R_d\oplus k)\},\\
 \T^n(S\modu)&=\{M\mid {}_{R_d}M\text{ is finite free}\}
                   \qquad(n\geq2).
\end{align*}
Moreover,
\[
 S\modu\supsetneq\T^1(S\modu)\supsetneq\T^2(S\modu)
 =\T^\infty(S\modu)=\GP(S\modu).
\]
The algebra $S$ is weakly Gorenstein but not Gorenstein.
For $1\leq q<m$, the module $M_q=S/(t^q)$ is nonprojective
Gorenstein projective and has a two-periodic complete resolution.
\end{corollary}

\begin{proof}
Corollary~\ref{cor:ring-transfer} and
Proposition~\ref{prop:base-algebra} give the formulas for $\T^n$.
The latter proposition also gives
$\GP(R_d\modu)=\Pj(R_d\modu)$.
The Gorenstein projective part of Corollary~\ref{cor:ring-transfer}
therefore gives $\T^2(S\modu)=\GP(S\modu)$.
Weak Gorensteinness follows from Theorem~\ref{thm:stabilization}.

For strictness, regard $R_d/(x_1)$ and $k$ as $S$-modules with
$t$ acting by zero. The first has dimension $d$ and is not
semisimple. It cannot be $R_d^a\oplus k^c$: its dimension forces
$a=0$, contradicting nonsemisimplicity. Hence it lies outside
$\T^1(S\modu)$. The module $k$ belongs to $\T^1(S\modu)$ but
not to $\T^2(S\modu)$.

Writing an element of $S$ as $\sum_{i=0}^{m-1}r_it^i$, multiplication
by $t$ and the $x_j$ shows that
\[
 \soc S=\mm t^{m-1}=\bigoplus_{j=1}^d kx_jt^{m-1}.
\]
Its dimension is $d>1$, so $S$ is not Gorenstein.

Finally, $M_q$ has the complete resolution
\[
 \cdots\xra{t^q}S\xra{t^{m-q}}S\xra{t^q}S
 \xra{t^{m-q}}S\xra{t^q}\cdots.
\]
Indeed, $\Ker(t^q)=t^{m-q}S$ and $\Ker(t^{m-q})=t^qS$.
The same identities hold on free modules and their direct summands,
so applying $\Hom_S(-,Q)$ for a projective $Q$ leaves this complex
exact. Multiplication by $t^{m-q}$ identifies $M_q$ with the
cocycle $t^{m-q}S$.
If $M_q$ were projective, the nonzero nilpotent ideal $(t^q)$ would
be a direct summand of $S$ and hence generated by a nonzero idempotent,
which is impossible.
\end{proof}

For $d=m=2$, this gives the six-dimensional algebra
\[
 \Lambda=k[x,y,t]/\big((x,y)^2,t^2\big).
\]
Its radical satisfies $J^2=(xt,yt)$ and $J^3=0$, and $\Lambda/(t)$
has a one-periodic complete resolution with differential
multiplication by $t$.

\medskip
\noindent\textbf{Acknowledgements.}

The author thanks Professor Xiao-Wu Chen, Zhi-Wei Li and Xiaojin Zhang for their helpful suggestions and comments. The author also thanks Professor G. Militaru for pointing out the non-standard citations in the manuscript and offering relevant references. This work is supported by the National Natural Science Foundation of China (No.
12371015, 12571034).

\end{document}